\documentclass[12pt]{amsart}
\usepackage{amsfonts,latexsym,amsthm,amssymb,graphicx}
\usepackage{wrapfig}
\usepackage[all]{xy}
\DeclareFontFamily{OT1}{rsfs}{}
\DeclareFontShape{OT1}{rsfs}{n}{it}{<-> rsfs10}{}
\DeclareMathAlphabet{\mathscr}{OT1}{rsfs}{n}{it}

\CompileMatrices

\newtheorem{theorem}{Theorem}[section]
\newtheorem{lemma}[theorem]{Lemma}
\newtheorem{corol}[theorem]{Corollary}
\newtheorem{prop}[theorem]{Proposition}
\newtheorem*{conj}{Conjecture}
\newtheorem*{ques}{Question}
{\theoremstyle{remark} \newtheorem{remark}[theorem]{Remark}
\newtheorem{example}[theorem]{Example}}
\newcommand{\Abb}{{\mathbb{A}}}
\newcommand{\Cbb}{{\mathbb{C}}}
\newcommand{\Lbb}{{\mathbb{L}}}
\newcommand{\Fbb}{{\mathbb{F}}}
\newcommand{\Nbb}{{\mathbb{N}}}
\newcommand{\Pbb}{{\mathbb{P}}}
\newcommand{\Qbb}{{\mathbb{Q}}}
\newcommand{\Tbb}{{\mathbb{T}}}
\newcommand{\Zbb}{{\mathbb{Z}}}
\newcommand{\cA}{{\mathscr A}}
\newcommand{\cL}{{\mathscr L}}
\newcommand{\cO}{{\mathscr O}}
\newcommand{\Til}[1]{{\widetilde{#1}}}
\newcommand{\one}{1\hskip-3.5pt1}
\newcommand{\csm}{{c_{\text{SM}}}}
\newcommand{\hA}{{\widehat{A}}}
\newcommand{\hH}{{\widehat{H}}}
\newcommand{\hcA}{{\widehat{\cA}}}
\newcommand{\qede}{\hfill$\lrcorner$}
\newcommand{\mustata}{{Musta\c{t}\v{a}}}
\newcommand{\decor}[1]{{{#1}'}}
\newcommand{\CP}{\chi}

\DeclareMathOperator{\rk}{rk}
\DeclareMathOperator{\codim}{codim}
\DeclareMathOperator{\Var}{Var}

\title[Grothendieck classes and Chern classes of hyperplane arrangements]{
Grothendieck classes and Chern classes of hyperplane arrangements}
\author{Paolo Aluffi}
\address{
Mathematics Department, 
Florida State University,
Tallahassee FL 32306, U.S.A.
}
\email{aluffi@math.fsu.edu}

\begin{document}

\begin{abstract}
We show that the characteristic polynomial of a hyperplane arrangement can be
recovered from the class in the Grothendieck group of varieties of the complement
of the arrangement. This gives a quick proof of a theorem of Orlik
and Solomon relating the characteristic polynomial with the ranks of the cohomology
of the complement of the arrangement.

We also show that the characteristic polynomial can be computed from the total 
Chern class of the complement of the arrangement. In the case of free arrangements,
we prove that this Chern class agrees with the Chern class of the dual of a bundle of 
differential forms with logarithmic poles along the hyperplanes in the arrangement;
this follows from work of  \mustata{} and Schenck. 
We conjecture that this relation holds for any locally quasi-homogeneous free divisor. 

We give an explicit relation between the characteristic polynomial of an arrangement 
and the Segre class of its singularity (`Jacobian') subscheme. This gives a variant of 
a recent result of Wakefield and Yoshinaga, and shows that the Segre class of
the singularity subscheme of an arrangement together with the degree of the arrangement
determine the ranks of the cohomology of its complement.

We also discuss the positivity of the Chern classes of hyperplane arrangements:
we give a combinatorial interpretation of this phenomenon, and we discuss the
cases of generic and free arrangements.
\end{abstract}

\maketitle


\section{Introduction}\label{intro}
\subsection{}
The guiding theme of this note is the relation between the combinatorics
of a projective hyperplane arrangement  and algebro-geometric and 
intersection-theoretic invariants of the union of the hyperplanes.

Our results are as follows. We work over a field $k$. We consider a hyperplane
arrangement $\cA$ in $\Pbb^n$, and the corresponding central arrangement
$\hcA$ in $k^{n+1}$. We denote by $A\subseteq \Pbb^n$ the union of the
hyperplanes in $\cA$. We let $\CP_{\hcA}(t)$ be the characteristic
polynomial of $\hcA$, and denote by $\underline{\CP_{\hcA}}(t)$ the quotient
$\CP_\hcA(t)/(t-1)$ (also a polynomial).
We let $M(\cA)$ be the complement $\Pbb^n \smallsetminus A$.

\begin{theorem}\label{Grothintro}
Let $\Lbb$ be the class of the affine line in the Grothendieck group of $k$-varieties 
$K(\Var_k)$. Then $\underline{\CP_\hcA}(\Lbb)$ evaluates the class of the 
complement $M(\cA)$ in~$K(\Var_k)$.
\end{theorem}

If the arrangement is defined over a finite field, Theorem~\ref{Grothintro}
shows how to recover the Poincar\'e polynomial of an arrangement by counting
points, an observation made by several authors (see e.g., Theorem~2.69 in
\cite{MR1217488}). Placing this elementary result in the Grothendieck ring 
$K(Var_k)$ for arbitrary $k$ has not-so-elementary applications: for example,
we note that the Orlik--Solomon theorem relating the characteristic polynomial 
to the Poincar\'e polynomial of the complement of an arrangement 
(\cite{MR558866}) is an immediate consequence of this result 
(Corollary~\ref{OSredux}).
Another application is the determination of the stable rational equivalence 
class of $M(\cA)$
(Corollary~\ref{stableb}).

\subsection{}
As observed in \cite{MR2504753}, Proposition~2.2, the information carried by the
{Gro\-then\-dieck} class of a union of linear subspaces is equivalent to the information
in the image of its `Chern--Schwartz--MacPherson' (CSM) class in $A_*\Pbb^n$. 
Therefore, Theorem~\ref{Grothintro} may be recast in terms of CSM classes:

\begin{theorem}\label{CSMint}
The Chern--Schwartz--MacPherson class of
the complement $M(\cA)$ may be obtained by
replacing $t^k$ with $[\Pbb^k]$ in $\underline{\CP_\hcA}(t+1)$.
\end{theorem}

For a reminder on Chern--Schwartz--MacPherson classes, see~\S\ref{CSMintro}.
J.~Huh has also observed that the additivity properties of these classes lead to 
formulas for the characteristic polynomial: see Remark~26 in~\cite{Huh}.
In the particular case of {\em free\/} arrangements, we prove the following:

\begin{theorem}\label{freeintro}
Let $\cA$ be a projective arrangement such that the corresponding affine arrangement
$\hcA$ is free. Then the Chern--Schwartz--MacPherson class of the complement
$M(\cA)$ equals $c(\Omega^1_{\Pbb^n}(\log A)^\vee)\cap [\Pbb^n]$.
\end{theorem}

We view Theorem~\ref{CSMint} as the primary result,
as it holds for arbitrary arrangements, and Theorem~\ref{freeintro} as a computation
in the particular case of free arrangements. We obtain Theorem~\ref{freeintro} as
a consequence of a result of \mustata{} and Schenck (\cite{MR1843320}).

It is natural to question whether a similar result may hold for a substantially more
general situation. The following conjecture appears to be consistent with all
known cases:

\begin{conj}
If $X$ is a locally quasi-homogeneous free divisor in a nonsingular variety~$V$, 
then $\csm(\one_{V\smallsetminus X})$ equals $c(\Omega^1_V(\log X)^\vee) \cap [V]$.
\end{conj}

This statement holds if $X$ is a divisor with simple normal crossings in $V$, 
and (as verified in Theorem~\ref{freeintro}) if $V=\Pbb^n$ and $X$ is a free hyperplane
arrangement. The restriction to locally quasi-homogeneous divisors is suggested by
the case $V=$ surface, studied by Xia Liao \cite{Liao}.
(We recall that a hypersurface is `locally quasi-homogeneous' if at each point
it admits a weighted homogeneous equation, with positive weights, with respect 
to some set of analytic parameters.)

\subsection{}
The subtlest part of the information carried by the Chern--Schwartz--MacPherson 
class of a hypersurface in a nonsingular variety amounts to the {\em Segre class\/}
of its singularity (`Jacobian') subscheme. Extracting this information and applying
Theorem~\ref{CSMint} gives the following
(in characteristic~$0$):

\begin{corol}\label{Segreintro}
The characteristic polynomial $\CP_\hcA(t)$ for a hyperplane arrangement $\cA$
in $\Pbb^n$ may be recovered from the degree of $A$ and
the image in $A_*\Pbb^n$ of the Segre class of the singularity subscheme of $A$.
\end{corol}

The precise relation between these invariants is given in~Theorem~\ref{segrethm}.
We note that Wakefield and Yoshinaga have proven (\cite{MR2424913}) that (essential) 
hyperplane arrangements may be recovered from their singularity subschemes.
Corollary~\ref{Segreintro} shows that the Segre class of the same scheme 
suffices in order to recover the most basic combinatorial information of the 
arrangement. Putting together Theorem~\ref{segrethm} and Corollary~\ref{OSredux}
yields the following surprisingly simple relation: for any hyperplane arrangement
in $\Pbb^n_\Cbb$ and $k\le n$
\begin{equation*}
\tag{*}
\rk H^k(M(\cA),\Qbb) = \sum_{i=0}^k \binom k i (d-1)^{k-i} \sigma_i\quad,
\end{equation*}
where $d$ is the degree of the arrangement and the integers $\sigma_i$ are
determined by the push-forward of the Segre class of the singularity subscheme~$S$:
\[
\sum_{i=0}^n \sigma_i \cap [\Pbb^{n-i}] = [\Pbb^n]-\iota_* s(S,\Pbb^n) \quad.
\]
This equality should be compared with Huh's formula expressing the Betti numbers 
of the complement as mixed multiplicities; see the proof of Corollary~25 in \cite{Huh}.
A referee points out that hyperplane arrangements are the only hypersurfaces
for which (*) holds. Indeed, the rank of $H^1$ of the complement of a
hypersurface is one less than the number of distinct irreducible components
(\cite{MR1194180}, Chapter~4, Proposition~1.3). If (*) holds for a degree~$d$
hypersurface, then the rank of $H^1$ equals $d-1$, and it follows that the
hypersurface consists of $d$ hyperplanes.

\subsection{}
Finally, we discuss positivity of Chern classes of hyperplane arrangements. 
Chern--Schwartz--MacPherson classes arising in 
combinatorial situations have a mysterious tendency to be effective. For example,
the Chern--Schwartz--MacPherson class of a toric variety is effective
(\cite{MR2209219}); CSM classes of Schubert varieties are conjecturally effective
(\cite{MR2448279}, \S4); and all CSM classes of graph hypersurfaces computed to 
date are effective (cf.~\cite{MR2504753}, Conjecture~1.5). It is natural to inquire 
about the positivity of CSM classes of hyperplane arrangements.

\begin{theorem}\label{positivintro}
\begin{itemize}
\item The CSM class of a generic hyperplane arrangement $X\subseteq \Pbb^n$ of
degree $d\le n+3$ (for $n$ even), resp.~$d\le n+4$ (for $n$ odd) is effective.
\item The CSM class of every free hyperplane arrangement of degree $\le n$
in $\Pbb^n$, $n\le 8$, is effective.
\end{itemize}
\end{theorem}

However, in \S\ref{posifree} we exhibit a free arrangement of degree~$9$ in 
$\Pbb^9$ that has non-effective Chern--Schwartz--MacPherson class. The effectivity 
of the CSM class of an arrangement has a straightforward combinatorial 
interpretation, which we give in Proposition~\ref{effe}.
\medskip

I thank an anonymous referee for valuable comments.


\section{Grothendieck classes}\label{grothsec}

\subsection{}\label{genera}
We work over a  field $k$. For considerations involving Chern--Schwartz--MacPherson 
classes, $k$ will be assumed to be algebraically closed, of characteristic~$0$
(see the comments at the beginning of \S\ref{CSMintro}).

We consider an arrangement $\cA$ of distinct hyperplanes $H_i$ in $\Pbb^n$, 
$i=1,\dots,d$, $d\ge 2$. We also consider the corresponding arrangement
$\hcA$ of hyperplanes $\hH_i\subset V=k^{n+1}$; this has the advantage of being 
{\em central.\/} While we are interested in the projective geometry of $\cA$, basic
definitions and results in the literature are more often given for the associated affine 
arrangement $\hcA$.
We will denote by $A$ and $\hA$, respectively the union of the hyperplanes in $\cA$
and $\hcA$, respectively. We will view~$A$ as a reduced singular hypersurface
of~$\Pbb^n$ of degree~$d\ge 2$.

We will be especially interested in the complements $M(\hcA)=V\smallsetminus \hA$,
$M(\cA)=\Pbb^n\smallsetminus A$. Note that $M(\hcA)$ is a trivial $k^*$-fibration over
$M(\cA)$.

We will denote by $L(\hcA)$ the poset of intersections $\cap_{i\in J} \hH_i$,
partially ordered by reverse inclusion. (For $x,y\in L(\hcA)$ we will write
interchangeably $y\supseteq x$ or $y\le x$ to denote that the subspace
$y$ contains the subspace $x$, i.e., that $y$ precedes $x$ in the poset $L(\hcA)$.)
The space $V$ itself is viewed as 
the intersection over $J=\emptyset$, and is denoted $0\in L(\hcA)$. As $\hcA$
is central, $L(\hcA)$ also has a maximum $1$, corresponding to $\cap_i \hH_i$.
The arrangement is {\em essential\/} if $\cap_i \hH_i$ is the origin; this assumption
will not be needed in this paper.

The {\em M\"obius function\/} of $L(\hcA)$ is defined on pairs $x\le y$ by the
following prescription:
\begin{align*}
\mu(x,x) &= 1 \quad\text{for all $x\in L(\hcA)$} \\
\sum_{x\le z\le y} \mu(x,z) &= 0 \quad\text{for all $x<y$ in $L(\hcA)$.}
\end{align*}
Write $\mu(x)$ for $\mu(0,x)$.
The {\em characteristic polynomial\/} of the arrangement $\hcA$ is
\[
\CP_\hcA(t):=\sum_{x\in L(\hcA)} \mu(x) t^{\dim x}=t^{n+1}+\cdots\quad.
\]
Note that $\CP_\hcA(1)=\sum_{0\le z\le 1} \mu(0,z)=0$. 
The {\em Poincar\'e polynomial\/} of the arrangement~is
\[
\pi_\hcA(t):=\sum_{x\in L(\hcA)} \mu(x) (-t)^{\codim x} = (-t)^{n+1}\cdot \CP_\hcA(-t^{-1})
\quad.
\]

In any extension in which $\CP_\hcA(t)$ factors completely, 
$\CP_\hcA(t)=(t-\alpha_1)\cdots (t-\alpha_{n+1})$, we have
\begin{align*}
\pi_\hcA(t) = (-t)^{n+1}\cdot \CP_\hcA(-t^{-1})
&=(-t)^{n+1} \left(-\frac 1t-\alpha_1\right) \cdots \left(-\frac 1t-\alpha_{n+1}\right) \\
& = (1+\alpha_1 t)\cdots (1+\alpha_{n+1} t)\quad.
\end{align*}
Note that $\pi_{\hcA}(-1)=\CP_{\hcA}(1)=0$: 
we may assume that $\alpha_{n+1}=1$, and we let
\[
\underline{\CP_\hcA}(t)=(t-\alpha_1)\cdots (t-\alpha_n)=\frac{\CP_\hcA(t)}{t-1}
\quad,\quad
\underline{\pi_\hcA}(t)=(1+\alpha_1 t)\cdots (1+\alpha_n t)= \frac{\pi_\hcA(t)}
{1+t}\quad.
\]
Thus, $\underline{\CP_\hcA}(t)$ and $\underline{\pi_\hcA}(t)$ are polynomials of degree 
at most $n$. The results in this paper are most naturally expressed in terms of these
polynomials. As we will see in a moment, 
$\underline{\pi_\hcA}(t)$ has nonnegative coefficients.

\subsection{}
Our first task is the computation of the Grothendieck class of the complement $M(\cA)$
in terms of the characteristic polynomial. We denote by $[X]$ the class of
a $k$-variety $X$ in the Grothendieck ring $K(\Var_k)$ of varieties.
The class of $\Abb^1_k$ is denoted~$\Lbb$.

\begin{theorem}\label{Gclassc}
With notation as above:
\[
[M(\cA)]=\underline{\CP_\hcA}(\Lbb)\quad.
\]
\end{theorem}

\begin{proof}
For $x\in L(\hcA)$, let
\[
x^\circ = x \smallsetminus \cup_{y> x} y
\]
be the complement of the union of smaller intersections. 
Write $[x]=\Lbb^{\dim x}$, $[x^\circ]$ for the classes in $K(\Var_k)$ of the corresponding
subsets of $V$. Then
\[
[y] = \sum_{x\ge y} [x^\circ]\quad.
\]
Applying M\"obius inversion (\cite{MR1217488}, Proposition~2.39), this gives
\[
[y^\circ] = \sum_{x\subseteq y} \mu(y,x) [x]
\]
and in particular
\[
[M(\hcA)] = [0^\circ] = \sum_{x\in L(\hcA)} \mu(0,x)[x]
=\sum_{x\in L(\hcA)} \mu(x)\Lbb^{\dim x} = \CP_\hcA(\Lbb)\quad.
\]
Since $M(\hcA)$ fibers over $M(\cA)$, with $k^*$ fibers, we have
$[M(\hcA)]=(\Lbb-1) \cdot [M(\cA)]$, and the stated formula follows.
\end{proof}

\subsection{}
If $\cA$ is defined over a finite field $\Fbb_q$, the content of Theorem~\ref{Gclassc}
is that the information carried by the characteristic polynomial of $\hcA$ may be 
recovered from counting points of $A$ over $\Fbb_{q^r}$. This observation is of 
course not new, see \cite{MR1409420} (e.g.~Theorem~2.2) for a thorough study of 
characteristic polynomials of arrangements defined over finite fields.

For complex arrangements, Theorem~\ref{Gclassc} has the following immediate
consequence:

\begin{corol}
With notation as above, the Deligne-Hodge polynomial of the complement
$M(\cA)$ of a complex hyperplane arrangement in $\Pbb^n$ equals
$\underline{\CP_\hcA}(uv)$.
\end{corol}
\noindent The Deligne-Hodge polynomial is the polynomial 
$\sum_{p,q} e^{p,q}(X) u^p v^q$, where 
$e^{p,q}(X)=\sum_k (-1)^k h^{p,q} (H_c^k(X,\Qbb))$. It is determined by the
Grothendieck class of $X$, by its well-known additivity/multiplicativity properties 
(\cite{MR873655}). As the polynomial for $\Lbb$ is $uv$, the statement follows
immediately from Theorem~\ref{Gclassc}.

Taking into account the fact the mixed Hodge structure on the cohomology of 
the complement of a hyperplane arrangement is pure (\cite{MR1131042}), we
obtain the following:

\begin{corol}\label{OSredux}
$\underline{\CP_\hcA}(t)=\sum_{k=0}^n (-1)^{n+k} \rk H_c^{n+k}(M(\cA),\Qbb)\,t^k$.
\end{corol}

\noindent (To keep track of weights: the Hodge structures of $H_c^{n+k}$ and $H^{n-k}$
are compatible through the Poincar\'e pairing, see~\cite{MR873655}, \S1.4 (f);
$H^{n-k}$ is of type $(n-k,n-k)$ by~\cite{MR1131042}; thus $H_c^{n+k}$
is of type $(k,k)$.)

As $\rk H_c^{n+k}(M(\cA),\Qbb)=\rk H^{n-k}(M(\cA),\Qbb)$, this statement is
equivalent to
\[
\underline{\pi_\hcA}(t)=\sum_{k=0}^n \rk H^k(M(\cA),\Qbb)\,t^k\quad.
\]
This shows that the coefficients of $\underline{\pi_\hcA}(t)$ are nonnegative. 
Since $M(\hcA)\cong M(\cA)\times k^*$, this also proves
\[
\pi_\hcA(t)=\sum_{k=0}^{n+1} \rk H^k(M(\hcA),\Qbb)\,t^k\quad,
\]
a classic result of Orlik and Solomon (\cite{MR1217488}, Theorem 5.93).
The approach presented here appears particularly straightforward, since it
shows that this result follows directly from Theorem~\ref{Gclassc}, which is 
a trivial consequence of M\"obius inversion.

\subsection{}
Another piece of information that may be derived from the Grothendieck class is the
stable birational equivalence class.
Two projective nonsingular irreducible
varieties $X$, $Y$ are {\em stably birational\/} if $X\times \Pbb^m$
is birational to $Y\times \Pbb^n$ for some $m$, $n$. Stable birational equivalence 
classes of complete nonsingular varieties generate a ring $\Zbb[SB]$, with addition 
defined by disjoint union and multiplication by product.
Every variety (possibly noncomplete or singular) has a well-defined class in 
$\Zbb[SB]$. The ring $\Zbb[SB]$ is defined and studied in \cite{MR1996804}. 

\begin{corol}\label{stableb}
Let $\cA$ be a complex hyperplane arrangement in $\Pbb^n$, and let $A\subset
\Pbb^n$ be the union of its components. 
Then the class of $A$ in $\Zbb[SB]$ equals 
$1-\underline{\CP_{\hcA}}(0)=1-(-1)^n \rk H^n(M(\cA),\Qbb)$.
\end{corol}

\begin{proof}
The ring $\Zbb[SB]$ is isomorphic to $K(\Var)/(\Lbb)$ (\cite{MR1996804}, 
Theorem~2.3 and Proposition~2.7). Setting $\Lbb=0$ in Theorem~\ref{Gclassc}
shows that $[M(\cA)]$ is congruent to the constant term of $\underline{\CP_\hcA}(0)$ in
$\Zbb[SB]$. Since $[\Pbb^n]=1$ in $\Zbb[SB]$, it follows that the stable
birational equivalence class of $A$ equals $1-\underline{\CP_\hcA}(0)\in \Zbb
\subseteq \Zbb[SB]$.
The equality $\underline{\CP_\hcA}(0)=(-1)^n \rk H^n(M(\cA),\Qbb)$ follows from
Corollary~\ref{OSredux}.
\end{proof}


\section{Chern--Schwartz--MacPherson classes}\label{CSMsec}

\subsection{}\label{CSMintro}
We now assume the ground field $k$ to be algebraically closed, of characteristic~$0$.
This assumption on the characteristic could likely be relaxed: in any environment
in which resolution of singularities is available, one may define a class satisfying 
inclusion-exclusion and the basic normalization property of 
Chern--Schwartz--MacPherson classes (see~\cite{MR2282409}, Definition~4.4 
and~\S3.3); these are the only tools needed in this section. However, 
characteristic~$0$ is necessary for the main covariance property of these classes
recalled below (see~\cite{MR2282409}, \S5.2), and some results on these classes
are currently only known in characteristic zero. Hence, we prefer to conservatively
adopt this assumption, at the price of limiting the scope of the results. (Many 
interesting arrangements can only be realized in positive characteristic.)

Recall that there is a homomorphism from the group of constructible functions on a variety
$X$ to the Chow group of $X$, $\varphi \mapsto c_*(\varphi)$, covariant with respect
to proper maps and such that $c_*(\one_X)$ equals the total Chern class of the tangent
bundle of~$X$ if $X$ is nonsingular (\cite{MR0361141}; Example~19.1.7 in \cite{85k:14004};
\cite{MR2282409}). The key covariance property of $c_*$ amounts to the fact that if
$\alpha: X\to Y$ is proper, and $\varphi$ is a constructible function, then
$c_*(\alpha_*(\varphi))=\alpha_*(c_*(\varphi))$. Here the push-forward of constructible
functions is defined by taking Euler characteristics of fibers.

We consider this theory over $X=\Pbb^n$. Every subvariety (possibly singular, noncomplete)
$Y\subseteq \Pbb^n$ determines a {\em Chern--Schwartz--MacPherson\/} (CSM) class 
$\csm(Y):=c_*(\one_Y)\in A_*\Pbb^n$.

\begin{theorem}\label{CSM}
Let $\cA$ be a hyperplane arrangement in $\Pbb^n$; denote by $A$ the union of the
hyperplanes in $\cA$, by $M(\cA)$ the complement of $A$ in $\Pbb^n$, and let
$\underline{\CP_\hcA}(t)$ be the polynomial introduced in \S\ref{grothsec}. Then
\[
\csm(M(\cA))=\underline{\CP_\hcA}(t+1)\quad,
\]
where the right-hand side is interpreted as a class in $A_*\Pbb^n$ by replacing
$t^k$ by $[\Pbb^k]$, $k=0,\dots,n$.
\end{theorem}

This result may be obtained by applying the formula relating CSM classes and 
Grothendieck classes for varieties obtained by elementary set-theoretic operations 
on linear subspaces (Proposition~2.2 in \cite{MR2504753}): the CSM class is 
obtained by viewing the Grothendieck class as a polynomial in $\Tbb=[k^*]$
and replacing $\Tbb^r$ with~$[\Pbb^r]$. Since $\Lbb=\Tbb+1$, 
Theorem~\ref{CSM} follows immediately from Theorem~\ref{Gclassc}.

In more conventional notation, Theorem~\ref{CSM} states the following:
\[
\csm(M(\cA))=\left(h^{n+1} \CP_{\hcA}\left(1+\frac 1h\right)\right)\cap [\Pbb^n]
=\left(h^n \underline{\CP_\hcA}\left(1+\frac 1h\right)\right)\cap [\Pbb^n]
\quad,
\]
where $h$ denotes the hyperplane class in $\Pbb^n$, and the expressions on the
right should be interpreted as the polynomials obtained by expanding them.
For the sake of completeness and ease of reference in later sections, we 
give here a direct proof of these formulas.

\begin{proof}
The second formula is just a restatement of the first one. To obtain the first one,
consider the functions from $L(\hcA)$ to the abelian group of constructible functions
on $k^{n+1}$, defined by $x \mapsto \one_x$, $x\mapsto \one_{x^\circ}$.
(Here $x^\circ$ is as in the proof of Theorem~\ref{Gclassc}.)
We have
\[
\one_y = \sum_{x\ge y} \one_{x^\circ}\quad,
\]
and hence
\[
\one_{y^\circ} = \sum_{x\subseteq y} \mu(y,x) \one_x
\]
by M\"obius inversion. In particular,
\[
\one_{M(\hcA)} =  \sum_{x\in L(\hcA)} \mu(x) \one_x\quad.
\]
For positive dimensional $x$, the characteristic functions $\one_x$ on $k^{n+1}$ are
pull-backs of the corresponding characteristic functions from $\Pbb^n$. It follows that
\[
\one_{M(\cA)} =  \sum_{x\in L(\hcA)} \mu(x) \one_{\underline x}
\]
on $\Pbb^n$, where $\underline x$ denotes the projective subspace corresponding to
$x\subset V$ (and $\underline x=\emptyset$ if $x$ is the origin in $V$).
Applying MacPherson's natural transformation, this shows that
\[
\csm(M(\cA)) =  \sum_{x\in L(\hcA)} \mu(x) \csm(\underline x)\in A_*\Pbb^n\quad.
\]
Now $\underline x\cong \Pbb^{\dim x-1}$, and hence
\[
\csm(\underline x) = (1+h)^{\dim x} h^{n+1-\dim x}\cap [\Pbb^n]
\]
as a class in $\Pbb^n$.
(In particular $\csm(\underline x)=0$ if $x$ is the origin in $V$, as it should, since
$h^{n+1}\cap [\Pbb^n]=0$.) It follows that
\[
\csm(M(\cA)) = \left(h^{n+1} \sum_{x\in L(\hcA)} \mu(x) \frac{(1+h)^{\dim x}}{h^{\dim x}}
\right) \cap [\Pbb^n]
=\left(h^{n+1} \CP_{\hcA}\left(1+\frac 1h\right)\right) \cap [\Pbb^n]\quad,
\]
as stated.
\end{proof}

\begin{corol}\label{poinc}
With the same notation:
\[
\csm(M(\cA))=(1+h)^n \underline{\pi_\hcA}\left(\frac {-h}{1+h}\right)\cap [\Pbb^n]
=\pi_{\hcA}\left(\frac {-h}{1+h}\right)\cap \left(c(T\Pbb^n)\cap [\Pbb^n]\right)\quad.
\]
\end{corol}

\begin{example}\label{norcro}
Suppose $\cA$ is {\em generic,\/} i.e., it consists of $d$ hyperplanes meeting with 
normal crossings. Then
\[
\csm(M(\cA))=c(\Omega(\log A)^\vee)\cap [A]=\frac 1{(1+h)^d}\cap
(c(T\Pbb^n)\cap [\Pbb^n])\quad,
\]
see e.g., Theorem~1 in \cite{MR2001d:14008}.
By Corollary~\ref{poinc}, then:
\[
\pi_{\hcA}\left(\frac {-h}{1+h}\right)=\frac 1{(1+h)^d}
\]
modulo $h^{n+1}$. Setting $t=\frac {-h}{1+h}$, i.e., $h=\frac {-t}{1+t}$, gives
\[
\pi_{\hcA}(t)\equiv (1+t)^d\mod t^{n+1}\quad.
\]
The coefficient of $t^{n+1}$ in $\pi_{\hcA}(t)$ is then determined by the fact
that $\pi_{\hcA}(-1)=0$.

For instance, $\pi_{\hcA}(t)=(1+t)^d$ for $d\le n+1$.
The {\em Boolean arrangement,\/} where $\hcA$ consists of the $n+1$ coordinate
hyperplanes in $V\cong k^{n+1}$, is of this type; cf.~\S2.3 in \cite{MR1217488}.
\qede\end{example}


\section{Free arrangements}\label{free}

\subsection{}
The key feature of Example~\ref{norcro} is the (well known) fact that if $D$ is a divisor 
with normal crossings
in a nonsingular variety, then the bundle $\Omega^1(\log D)$ is locally free and its
total Chern class computes the Chern--Schwartz--MacPherson class of the complement
of $D$. Thus, for normal crossings arrangements the left-hand side of the formulas 
in Theorem~\ref{CSM} and Corollary~\ref{poinc} may be viewed as the Chern class
of a bundle, and the formulas may be interpreted as an alternative computation of 
this class.

In this section we show that this interpretation extends to {\em free arrangements,\/}
by which we mean projective arrangements $\cA$ such that the corresponding
affine arrangements $\hcA$ are free in the sense of \cite{MR1217488}.
For these arrangements, the sheaf $\Omega^1_{\Pbb^n}(\log A)$ of differential 
$1$-forms with logarithmic poles along the union $A$ of the hyperplanes in $\cA$ 
is locally free. We will prove:

\begin{theorem}\label{freethm}
For free arrangements $\cA$ in $\Pbb^n$,
\begin{equation*}
\tag{$\dagger$}
\csm(M(\cA))=c(\Omega^1_{\Pbb^n}(\log A)^\vee)\cap [\Pbb^n]\quad.
\end{equation*}
\end{theorem}

We recall that sections of the logarithmic sheaf $\Omega^1_{\Pbb^n}(\log A)$ 
are differential forms~$\omega$ such that both $f\omega$ and $f d\omega$ are
regular, where $f=0$ is the equation for $A$. This definition, due to Saito, 
generalizes Deligne's definition for normal crossing divisors. In general, a divisor 
$D$ of a variety $X$ is said to be `free' if this sheaf is locally free.

\subsection{}
Taken together, Theorem~\ref{CSM} and Theorem~\ref{freethm} give a relation
between the characteristic polynomial of an arrangement $\cA$ and the 
total Chern class 
of $\Omega^1_{\Pbb^n}(\log A)$ in the case of free arrangements. 
Theorem~\ref{CSM} may be viewed as a generalization of this relation, 
as it holds for arbitrary projective arrangements---notwithstanding the fact that its
proof is completely trivial, while the proof of the particular case of free arrangements 
in the form of Theorem~\ref{freethm} requires some actual work. By our good fortune, 
the main ingredient in this work may be found in a paper of M.~\mustata{} and 
H.~Schenck; Theorem~\ref{freethm} will be obtained as a consequence of Theorem~4.1
in \cite{MR1843320}. 

This is one instance in which MacPherson's functorial theory of Chern 
classes clearly provides the `right' generalization of Chern classes to noncomplete 
(and/or singular) varieties, and it is tempting to guess that the equality in 
Theorem~\ref{freethm} may hold for more general free divisors. Work of Xia Liao
(\cite{Liao}) shows that for a reduced curve $X$ in a nonsingular surface $V$, the 
Chern--Schwartz--MacPherson class of the complement equals the Chern 
class of the corresponding bundle of logarithmic derivations only if the Milnor
and Tjurina numbers of the singularities of $X$ agree. This indicates that a hypothesis
of local quasi-homogeneity is likely necessary for a generalization of 
Theorem~\ref{freethm}; we conjecture as much in the introduction.

In any case, the question of comparing the CSM 
class of the complement of a divisor $D$ and the Chern classes of the corresponding 
sheaf of differential forms with logarithmic poles along $D$ appears to be interesting 
and approachable. Theorem~5.13 in \cite{denhamschulze} indicates that a correction
term will be necessary if this sheaf is not locally free, as it provides such a term for
the corresponding generalization of Theorem~4.1 in \cite{MR1843320}
to the `locally tame' case.

\subsection{}\label{prelim}
The formula in Theorem~\ref{freethm} holds if $A$ is any divisor with simple 
normal crossings in a nonsingular variety. For a proof of this elementary fact, 
see e.g., Theorem~1 in \cite{MR2001d:14008} or Proposition 15.3 in \cite{MR1893006};
this observation may in fact be used to give an alternative treatment of CSM classes
(\cite{MR2282409}). The obvious strategy to prove Theorem~\ref{freethm} would 
therefore be to apply resolution of singularities and reduce to the case of normal crossing 
divisors. The behavior of the left-hand side of $(\dagger)$ through a resolution is 
controlled by the covariance of CSM classes:

\begin{lemma}\label{trilem}
Let $X$ be a variety, and let $Y\subseteq X$ be a subscheme. Let $\rho: \Til X \to X$
be a proper map, and $Y'\subseteq \Til X$ any subscheme such that $\rho$ restricts 
to an isomorphism of the complements $M(Y')$ of $Y'$ in $\Til X$ and $M(Y)$ of 
$Y$ in $X$. Then 
\[
\csm(M(Y))=\rho_* \csm(M(Y'))\quad.
\]
\end{lemma}

\begin{proof}
Under the hypotheses specified in the statement, $\rho_*(\one_{M(Y')})=\one_{M(Y)}$,
and the equality follows then immediately from the covariance property of 
Chern--Schwartz--MacPherson classes (recalled in \S\ref{CSMintro}).
\end{proof}

Lemma~\ref{trilem} reduces the proof of Theorem~\ref{freethm} to verifying that
the Chern class of the bundle of differential forms with logarithmic poles is preserved
under push-forward for suitable blow-ups. The difficulty lies in the fact that the 
bundle itself is {\em not\/} preserved under blow-ups. 
However, Silvotti (\cite{MR1484696}) has analyzed the behavior of the logarithmic 
bundle under blow-ups in the case of arrangements, and we feel that his analysis
should suffice in order to obtain a proof of Theorem~\ref{freethm}.

In fact, Silvotti proves (\cite{MR1484696}, Proposition~4.5) that if 
$\Omega^1_X(\log D)$ splits as a direct sum of line bundles $\cL_1,\dots,\cL_n$, 
then the corresponding bundle $\Omega^1_{\decor X}(\log \decor D)$ in the blow-up
along a subvariety~$Y$ also splits, and in fact
\[
\Omega^1_{\decor X}(\log \decor D) \cong
(\sigma^*\cL_1 \otimes (-\mu_1 E)) \oplus \cdots \oplus 
(\sigma^*\cL_n \otimes (-\mu_n E)) 
\]
for non-negative integers $\mu_1,\dots,\mu_n$, where $E$ denotes the exceptional
divisor. (Terao proved that the splitting does occur in the case of free hyperplane 
arrangements, cf.~Proposition~5.1 in \cite{MR1484696}.) 
A proof of Theorem~\ref{freethm} follows easily if one could show that
{\em at most $\codim Y-1$ of the numbers $\mu_i$ are nonzero.\/}

In fact, given that Theorem~\ref{freethm} does hold, it seems that this must indeed 
be the case. It would be nice to have a direct proof of this fact.


\subsection{}\label{compa}
Theorem~4.1 in~\cite{MR1843320} provides us with an alternative approach to
Theorem~\ref{freethm}. Denote by $\Omega^1$ the {\em module\/} of differential
forms with logarithmic poles along the central arrangement $\hcA$ in $k^{n+1}$.
This is a graded module, hence it defines a coherent sheaf $\Til{\Omega^1}$ on
$\Pbb^n$. Under the assumption that the arrangement is free, $\Til{\Omega^1}$
is a rank-$(n+1)$ locally free sheaf on $\Pbb^n$. 

\begin{theorem}[\mustata-Schenck]\label{MusSch}
If $\hcA$ is an essential free arrangement, then
\[
c(\Til{\Omega^1})=\pi_\hcA (h)\quad,
\]
where $h=c_1(\cO_{\Pbb^n}(1))$.
\end{theorem}

\begin{remark}
In fact, \mustata{} and Schenck prove the equality in Theorem~\ref{MusSch}
under the weaker hypothesis that the arrangement is {\em locally\/} free.
Also, we note that while the statement of our Theorem~\ref{freethm} assumes 
the ground field to be algebraically closed of characteristic~$0$,
this assumption is not needed in the result of \mustata{} and Schenck.

In Theorem~\ref{MusSch}, the right-hand side should be parsed as the truncation
of the Poincar\'e polynomial modulo $h^{n+1}$, as $h^{n+1}=0$ in $\Pbb^n$.
Also, we recall that `essential' means that the intersection of all hyperplanes in 
the projective arrangement $\cA$ is empty.
\qede
\end{remark}

We now prove Theorem~\ref{freethm} as a corollary of Theorems~\ref{CSM}
and~\ref{MusSch}.

\begin{lemma}\label{exseq}
Let $\cA$ be a free arrangement. Then there is an exact sequence
\[
\xymatrix{
0 \ar[r] & \Omega^1_{\Pbb^n}(\log A) \ar[r] &
\Til{\Omega^1}\otimes \cO_{\Pbb^n}(-1) \ar[r] & \cO_{\Pbb^n} \ar[r] & 0\quad.
}
\]
\end{lemma}

\begin{proof}
The Euler derivation 
$x_0 \frac{\partial}{\partial x_0} + \cdots + x_n \frac{\partial}{\partial x_n}$ defines 
an epimorphism $\Til{\Omega^1} \to \cO_{\Pbb^n} (1)$ (cf.~Proposition~4.27 in
\cite{MR1217488} for the dual statement). The shifted epimorphism
$\Til{\Omega^1}(-1) \to \cO_{\Pbb^n}$ is the natural extension of the standard 
epimorphism $\cO_{\Pbb^n}(-1)^{\oplus(n+1)} \to \cO_{\Pbb^n}$
whose kernel defines the sheaf of differential forms over $\Pbb^n$ (as in 
\cite{MR0463157}, II.8.13). The kernel of this epimorphism is then the sheaf
of meromorphic differential forms satisfying the same conditions as the forms
in $\Omega^1$, and this is the definition of $\Omega^1_{\Pbb^n}(\log A)$.
\end{proof}

\begin{proof}[Proof of Theorem~\ref{freethm}]
First, we note that we may assume the arrangement to be essential. Indeed,
every arrangement $\cA$ in $\Pbb^n$ is a cone over an essential arrangement 
$\cA'$ in $\Pbb^{n-k}$, for some $k\ge 0$. Assume the formula ($\dagger$)
in Theorem~\ref{freethm} is known for $\cA'$:
\[
\csm(M(\cA'))=c(\Omega^1_{\Pbb^{n-k}}(\log A')^\vee)\cap [\Pbb^{n-k}]\quad.
\]
Inductively, it suffices to show that the formula for $\cA$ follows from this for $k=1$.
Write $\csm(M(\cA'))=g(h)\cap [\Pbb^{n-1}]$, for a polynomial $g$ of degree $\le n-1$;
we are assuming that $c(\Omega^1_{\Pbb^{n-1}}(\log A')^\vee)=g(h)$.
We have $\csm(A')=f(h)\cap [\Pbb^{n-1}]$ for $f(h)=(1+h)^n-h^n-g(h)$, hence
by Proposition~5.2 in \cite{MR2504753}
\[
\csm(A) = (1+h) f(h)\cap [\Pbb^n] + [\Pbb^0]\quad,
\]
and therefore
\begin{align*}
\csm(M(\cA))&=\csm(\Pbb^n)-\csm(A) \\
&=\big(((1+h)^{n+1}-h^{n+1})-((1+h) ((1+h)^n-h^n-g(h)) +h^n)\big) \cap [\Pbb^n] \\
&=(1+h) g(h) \cap [\Pbb^n]
\end{align*}
\noindent (The more combinatorially minded reader may reach the same conclusion 
as a consequence of Theorem~\ref{CSM}.)
On the other hand, by Lemma~\ref{exseq} we have
\[
c(\Omega^1_{\Pbb^n}(\log A')) = c(\Til {{\Omega'}^1}(-1))
\]
where ${\Omega'}^1$ is the module of differentials with logarithmic poles along $\cA'$.
Under the assumption that $\cA$ is free so is $\cA'$, and $\Omega^1={\Omega'}^1
\oplus k$. Therefore $c(\Til{\Omega^1}(-1))=(1-h)c(\Til{{\Omega'}^1}(-1))$, and again
by Lemma~\ref{exseq} we get
$c(\Omega^1_{\Pbb^n}(\log A))= (1-h) c(\Til {{\Omega'}^1}(-1))$, and hence
\[
c(\Omega^1_{\Pbb^n}(\log A))^\vee= (1+h)\, c(\Til{{\Omega'}^1}(-1)^\vee)
=(1+h)\, g(h)\quad.
\]
It follows that
\[
\csm(M(\cA))=(1+h) \,g(h)\cap [\Pbb^n]=c(\Omega^1_{\Pbb^n}(\log A))^\vee
\cap [\Pbb^n]\quad,
\]
which is ($\dagger$) for $\cA$, as claimed.

Therefore, we may assume that the arrangement is essential. By Theorem~\ref{MusSch},
\[
c(\Til{\Omega^1})=\pi_\hcA (h)=(1+h)\, \underline{\pi_\hcA}(h)
\]
in $A^*\Pbb^n$ (that is, modulo $h^{n+1}$). Using Lemma~\ref{exseq}, it follows that
\[
\underline{\pi_\hcA}(h) \equiv (1+h)^{-1} c(\Til{\Omega^1})\mod h^{n+1}
\equiv c(\Omega^1_{\Pbb^n}(\log A)\otimes \cO_{\Pbb^n}(1)) \mod h^{n+1}\quad,
\]
and hence
\[
\underline{\pi_\hcA}(h) = c(\Omega^1_{\Pbb^n}(\log A)\otimes \cO_{\Pbb^n}(1))
\]
as polynomials in $h$, since both sides have degree $\le n$.
Now (as in \S\ref{genera}) we factor $\underline{\CP_\hcA}(t)=(t-\alpha_1)\cdots (t-\alpha_n)$
over an extension\footnote{According to a theorem of Terao, the polynomial of a free
arrangement actually factors over $\Zbb$; cf.~\S\ref{posifree}. This is not needed here.} 
of $\Qbb$, 
and note that $\underline{\pi_\hcA}(t) = (1+\alpha_1 t)\cdots (1+\alpha_n t)$.
With this notation, we have shown
\[
c(\Omega^1_{\Pbb^n}(\log A)\otimes \cO_{\Pbb^n}(1))
=(1+\alpha_1 h)\cdots (1+\alpha_n h)\quad,
\]
and it follows that
\begin{align*}
c(\Omega^1_{\Pbb^n}(\log A)^\vee) &=(1-\alpha_1 h+h)\cdots (1-\alpha_n h+h) \\
&=h^n\left(1+\frac 1h-\alpha_1\right) \cdots \left(1+\frac 1h-\alpha_n\right) \\
&=h^n \underline{\CP_\hcA}\left(1+\frac 1h\right)
\end{align*}
with the usual caveat that the right-hand side must be interpreted as the polynomial
obtained by expanding it.
By Theorem~\ref{CSM} this proves
\[
c(\Omega^1_{\Pbb^n}(\log A)^\vee) \cap [\Pbb^n]
=\csm(M(\cA))\quad,
\]
and we are done.
\end{proof}

\begin{remark}
The projective version of Theorem~4.1 from~\cite{MR1843320} used above is also given 
in~\cite{denhamschulze}, \S5, and generalized to locally tame arrangements.
\qede
\end{remark}


\section{Segre classes of singularity subschemes}\label{segresec}

\subsection{}
In~\cite{MR2424913}, M.~Wakefield and M.~Yoshinaga prove that any (essential)
projective arrangement $\cA$ may be reconstructed from the {\em singularity
subscheme\/} $S$ of the hypersurface $A\subseteq \Pbb^n$, that is, the subscheme
defined by the partial derivatives of an equation for~$A$. In this section we prove 
that the polynomial $\underline{\pi_\hcA}(t)$ determines and is determined 
by the degree of the arrangement and the Segre class (cf.~\cite{85k:14004}, 
Chapter~4) of the singularity subscheme in $\Pbb^n$.
As we will use Chern--Schwartz--MacPherson classes for this result, we still work
over algebraically closed fields of characteristic~$0$. The following statement
makes sense over any field, but we do not know whether it holds in such
generality.

\begin{theorem}\label{segrethm}
Let $\iota:S\hookrightarrow \Pbb^n$ be the singularity subscheme of an arrangement 
$\cA$ in $\Pbb^n$. (That is, $S$ is defined by the partial derivatives of an equation 
for the hypersurface~$A$.) Let $\sigma_i\in \Zbb$ be such that
\[
[\Pbb^n]-\iota_* s(S,\Pbb^n) = \sum_{i=0}^n \sigma_i h^i \cap [\Pbb^n]  \in A_*\Pbb^n\quad.
\]
Then
\[
\underline{\pi_\hcA}(t) = \sum_{k=0}^n
\left(\sum_{i=0}^k \binom k i (d-1)^{k-i} \sigma_i\right) t^k\quad.
\]
\end{theorem}

\noindent Matching this formula with the expression for $\underline{\pi_\hcA}(t)$ obtained 
in the wake of Corollary~\ref{OSredux} yields the formula given in the introduction
for the ranks of the cohomology of the complement. 

\begin{proof}
By Corollary~\ref{poinc},
\begin{equation*}
\tag{$\dagger$}
\csm(A)=c(T\Pbb^n)\cap [\Pbb^n] - (1+h)^n 
\underline{\pi_\hcA}\left(\frac{-h}{1+h}\right) \cap [\Pbb^n]\quad.
\end{equation*}
On the other hand, by Theorem~I.4 in \cite{MR2001i:14009},
\[
\csm(A)=c(T\Pbb^n)\cap \left(\frac{d h}{1+dh}\cap [\Pbb^{n-1}]
+\frac 1{1+dh} \cap \big((\iota_* s(S,\Pbb^n))^\vee \otimes_{\Pbb^n} \cO(dh)\big)\right)
\]
where $d$ is the degree of the arrangement, and this expression uses notation
given in~\cite{MR2001i:14009}, \S1.4. Writing $\iota_* s(S,\Pbb^n)=\sum_{i=0}^n
s_i [\Pbb^i]$, this means
\begin{equation*}
\tag{$\ddagger$}
\csm(A)=(1+h)^{n+1}\left(\frac{d h}{1+dh}
+\frac 1{1+dh} \sum_{i=0}^n \frac{s_i\cdot (-h)^{n-i}}{(1+dh)^{n-i}}
\right)\cap [\Pbb^n]
\end{equation*}
Comparing $(\dagger)$ and $(\ddagger)$ gives the following equality of series
modulo $h^{n+1}$:
\[
\underline{\pi_\hcA}\left(\frac{-h}{1+h}\right)
=\frac {1+h}{1+dh} - \frac {1+h}{1+dh} 
\sum_{i=0}^n \frac{s_i\cdot (-h)^{n-i}}{(1+dh)^{n-i}}\quad,
\]
and setting $t=-h/(1+h)$ yields
\[
\underline {\pi_\hcA}(t) \equiv \frac 1{1-(d-1)\, t}\left(
1- \sum_{i=0}^n s_i\cdot \left(\frac t{1-(d-1)\, t}\right)^{n-i}\right)
\mod t^{n+1}\quad.
\]
Now with notation as in the statement we have $\sigma_0=1$ and $\sigma_i=
-s_{n-i}$ for $i>0$, and therefore
\[
\underline {\pi_\hcA}(t) \equiv \frac 1{1-(d-1)\, t} \,
\sum_{i=0}^n \sigma_i\cdot \left(\frac t{1-(d-1)\, t}\right)^i
\mod t^{n+1}\quad.
\]
The statement follows immediately from this equality.
\end{proof}

Using notation as in \S1.4 of~\cite{MR2001i:14009}, the formula given in 
Theorem~\ref{segrethm} may be rewritten as
\begin{equation*}
\tag{*}
\underline {\pi_\hcA}(h)\cap [\Pbb^n]
=\frac 1{1-(d-1) h}\cap \big(([\Pbb^n]-\iota_*s(S,\Pbb^n))\otimes_{\Pbb^n} \cO(-(d-1)h)\big)
\quad.
\end{equation*}
This is occasionally convenient in concrete computations, see~e.g., Example~\ref{book}.

\subsection{}
We illustrate Theorem~\ref{segrethm} with a few examples, in which the computation
of the Poincar\'e polynomial of the arrangement can also be performed easily with 
standard techniques. Algorithms computing Segre classes may be implemented in 
software systems such as Macaulay2 (\cite{M2}); one such implementation is 
described in~\cite{MR1956868}. See Example~\ref{counter} for an illustration of 
the use of such a routine.

\begin{example}
The three transversal intersections of the configuration of Example~\ref{fourlinesex}
count for one point each in the Segre class of the singularity subscheme. To
evaluate the contribution of the triple intersection, write it in local coordinates as
the singularity subscheme of $xy(x+y)=0$; the Jacobian ideal is then
\[
(2xy+y^2,x^2+2xy)
\]
and it follows that the contribution to the Segre class is $4$ points.
Thus $(\sigma_0,\sigma_1,\sigma_2)=(1,0,-7)$, and Theorem~\ref{segrethm}
gives
\[
\underline{\pi_\hcA}(t)=1+3t+t^2\quad.
\]
Therefore $\pi_\hcA(t)=(1+t)\underline{\pi_\hcA}(t)=1+4t+5t^2+2t^3$, i.e.,
$\CP_\hcA(t)=t^3-4t^2+5t-2$, as it should.
\qede
\end{example}

\begin{example}
Let $\cA$ consist of three planes in $\Pbb^3$, with equation $xyz=0$.
\end{example}
\begin{wrapfigure}{l}{0.45\textwidth}
  \begin{center}
    \includegraphics[width=0.3\textwidth]{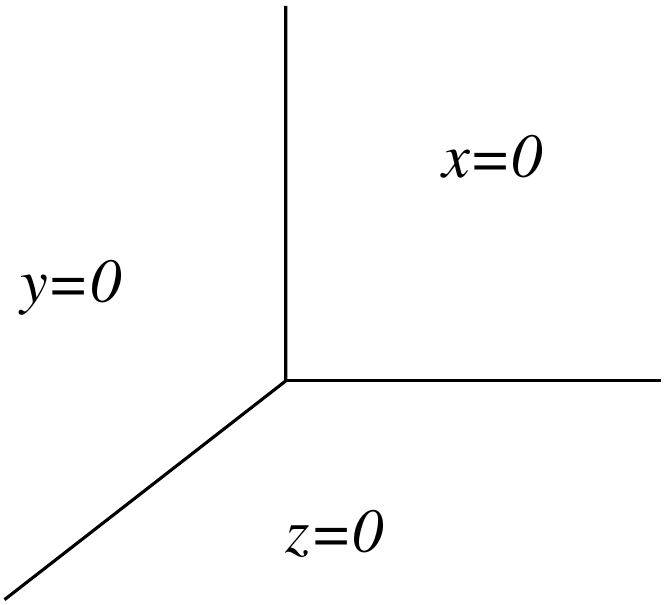}
  \end{center}
\end{wrapfigure}
The singularity subscheme $S$ is supported on three lines; it is defined by the 
ideal $(yz,xz,xy)$, so it is the intersection of three quadrics $Q_1,Q_2,Q_3$. 
The Segre class $s(S,\Pbb^3)$ is $3[\Pbb^1]+m[\Pbb^0]$ for some integer~$m$. 
One way to evaluate $m$ is
the following: the intersection product of the three quadrics must be $8$
by B\'ezout's theorem, and can be evaluated by applying the `basic
construction' (Proposition~6.1 (a) in \cite{85k:14004}) to the fiber diagram
\[
\xymatrix{
S=Q_1\cap Q_2\cap Q_3 \ar[r] \ar[d] & \Pbb^3 \ar[d] \\
Q_1\times Q_2\times Q_3 \ar[r] & \Pbb^3\times \Pbb^3\times \Pbb^3
}
\]
Writing $h_i$ for the hyperplane class in the $i$-th copy of $\Pbb^3$,
this gives
\[
8=\int (1+2h_1)(1+2h_2)(1+2h_3)\cap (3[\Pbb^1]+m[\Pbb^0]) = 18+m\quad,
\]
from which $m=-10$.

With notation as in Theorem~\ref{segrethm} we have $(\sigma_0,\dots,\sigma_3)
=(1,0,-3,10)$, from which
\[
\underline{\pi_\hcA}(t)=1+2t+(4-3)t^2+(8-3\cdot 2 \cdot 3+10)t^3 = 1+2t+t^2=(1+t)^2\quad.
\]
Therefore $\pi_\hcA(t)=(1+t)^3$. Of course this agrees with Example~\ref{norcro}, since
$\cA$ is a generic arrangement.
\qede
\medskip

\begin{example}\label{book}
Let $\cA$ consist of $d$ hyperplanes in the pencil of hyperplanes containing a fixed
codimension-$2$ subspace in $\Pbb^n$.
\end{example}
\begin{wrapfigure}{l}{0.45\textwidth}
  \begin{center}
    \includegraphics[width=0.28\textwidth]{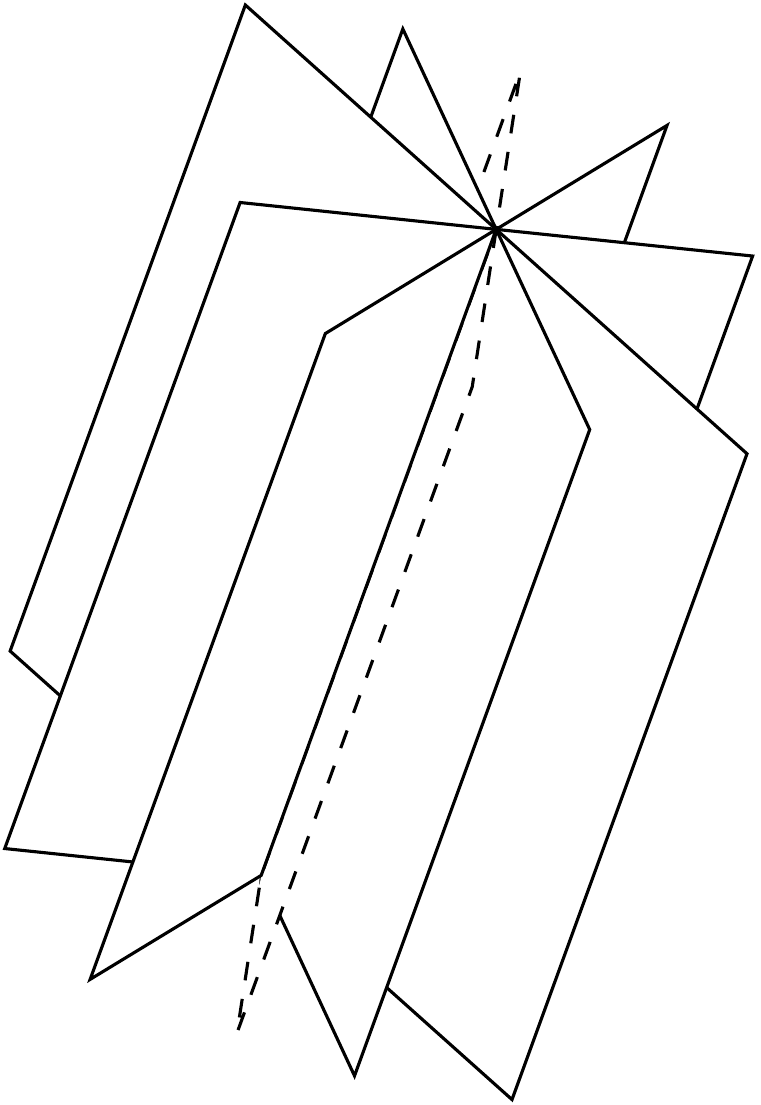}
  \end{center}
\end{wrapfigure}
The singularity subscheme $S$ is supported on $\Pbb^{n-2}$. To evaluate its
Segre class, blow-up along this subspace; if $E$ is the exceptional divisor,
the Segre class of the latter pushes forward to the Segre class of $\Pbb^{n-2}$,
by the birational invariance of Segre classes:
$
\frac{E}{1+E} \mapsto s(\Pbb^{n-2},\Pbb^n) = \frac 1{(1+h)^2}\cap [\Pbb^{n-2}]
$.
A straightforward computation shows that the singularity subscheme pulls
back to $(d-1)$ times the exceptional divisor. Therefore (again by birational
invariance) $s(S,\Pbb^n)$ is the push-forward of $(d-1)E/(1+(d-1)E)$, and
matching terms gives
\[
\iota_* s(S,\Pbb^n) = \frac 1{(1+(d-1) h)^2}\cap (d-1)^2 [\Pbb^{n-2}]\quad.
\]
Using (*), we get that
\begin{align*}
\underline{\pi_\hcA}(h)\cap [\Pbb^n] &=
\frac 1{1-(d-1)h}\cap \left([\Pbb^n]-\frac{(d-1)[\Pbb^{n-2}]}{(1+(d-1)h)^2}
\otimes \cO(-(d-1)h)\right) \\
&=\frac 1{1-(d-1)h}\cap \left([\Pbb^n]-(d-1)^2 [\Pbb^{n-2}]\right) \\
&=(1+(d-1)h)\cap [\Pbb^n]\quad.
\end{align*}
Therefore $\underline{\pi_\hcA}(t)=1+(d-1)t$. It follows that
$\pi_{\hcA}(t)=(1+(d-1)t)(1+t)=1+dt+(d-1)t^2$, and $\CP_\hcA(t)=t^{n+1}-d t^n+(d-1)t^{n-1}$.
(This is of course also evident from the poset associated with this arrangement.)
\qede
\medskip


\section{Positivity}\label{posit}

\subsection{}
One problem that prompted us to take a more careful look at hyperplane arrangements
is the issue of {\em positivity\/} of Chern--Schwartz--MacPherson classes. In the nonsingular
case, positivity of Chern classes is well understood; for example, $c(TX)\cap [X]$ is 
effective if $TX$ is generated by global sections (cf.~\cite{85k:14004}, Example~12.1.7).
We know of no such statement for Chern classes of singular varieties, and preciously
few examples are known: the CSM class of a toric variety is represented by an effective 
cycle (this follows from ``Ehlers' formula'', see e.g.,~\cite{MR1197235}), and CSM 
classes of Schubert varieties are conjecturally effective. It is natural to ask the following

\begin{ques}
Denote by $A\subseteq \Pbb^n$ the union of the hyperplanes of an arrangement $\cA$.
For which arrangements $\cA$ is $\csm(A)$ effective?
\end{ques}

Here, by `effective' we mean that $\csm(A)\in A_*\Pbb^n$ should be represented by 
an effective cycle. By Theorem~\ref{CSM}, $\csm(A)$ is determined by the 
characteristic polynomial of $\hcA$, so this is a combinatorial question.

\subsection{}
Here is the explicit translation of effectivity in combinatorial terms:

\begin{prop}\label{effe}
Let $\cA$ be a hyperplane arrangement in~$\Pbb^n$, and let $\mu$ be the
corresponding M\"obius function, as in \S\ref{genera}. Then $\csm(A)$ is
effective if and only if all coefficients of the polynomial
\[
-\sum_{x\ne 0} \mu(x) (t+1)^{\dim x}
\]
are nonnegative.
\end{prop}

In fact, the coefficient of $t^k$ in this expression equals the coefficient of
$[\Pbb^{k-1}]$ in $\csm(A)$, for $k\ge 1$; the constant term equals~$1$.

\begin{proof}
In the (direct) proof of Theorem~\ref{CSM} we obtained the equality
\[
\csm(M(\cA)) =  \sum_{x\in L(\hcA)} \mu(x) \csm(\underline x)\quad,
\]
where $\underline x$ denotes the projective subspace of $\Pbb^n$
determined by $x$. The summand corresponding to $x=0$ is 
$\mu(0) \csm(\underline 0)=c(T\Pbb^n)\cap [\Pbb^n]$. Thus
\[
\csm(A)=c(T\Pbb^n)\cap [\Pbb^n] - \csm(M(\cA))
= -\sum_{x\ne 0} \mu(x) \csm(\underline x)\quad.
\]
Now $\underline x\cong \Pbb^{\dim x-1}$, so
\[
\csm(\underline x) = \sum_{k=1}^{\dim x} \binom{\dim x}{k} [\Pbb^{k-1}]\quad.
\]
Hence the coefficient of $[\Pbb^{k-1}]$ in $\csm(A)$ equals
\[
-\sum_{x\ne 0\,:\, \dim x\ge k} \mu(x)\binom{\dim x}k\quad,
\]
that is, the coefficient of $t^k$ in
\[
-\sum_{x\ne 0} \mu(x) (t+1)^{\dim x}\quad.
\]
This holds for $k\ge 1$. On the other hand, since $\sum_x \mu(x)=0$ and $\mu(0)=1$,
the constant term in this expression is $1>0$. It follows that $\csm(A)$ is effective
if and only if all coefficients of this polynomial are nonnegative, which is the
statement.
\end{proof}

\begin{example}\label{fourlinesex}
We illustrate Proposition~\ref{effe} with a simple example.
\begin{center}
\includegraphics[scale=.5]{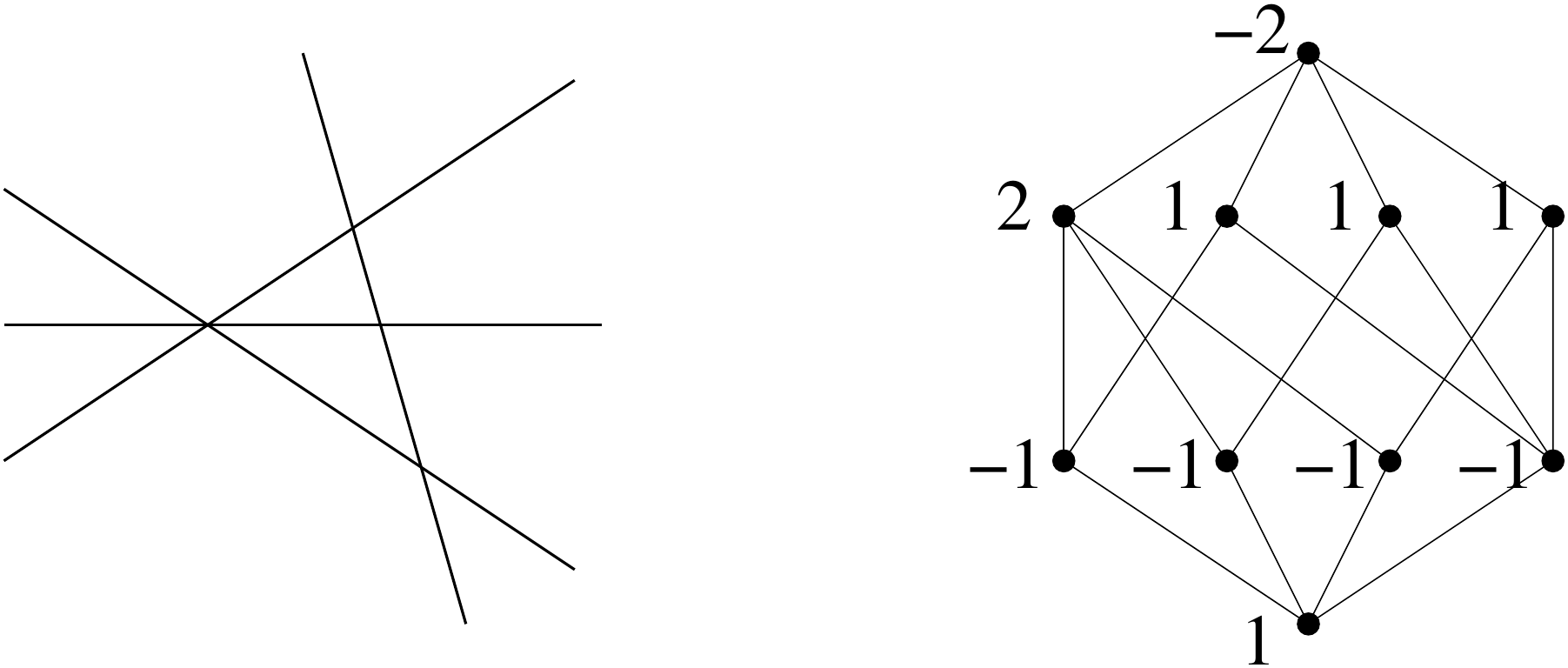}
\end{center}
The poset and M\"obius function for the arrangement on the left (in $\Pbb^2$) 
are given on the right; the $0$ element of the lattice, i.e., $k^3$, is at the bottom.
The polynomial appearing in Proposition~\ref{effe} is
\[ 
-(-4 (t+1)^2 + 5(t+1) -2) = 4 t^2 +3t+1\quad.
\]
As the coefficients are all positive, the CSM class of this arrangement is effective;
this class equals $4[\Pbb^1]+3 [\Pbb^0]$.
\qede
\end{example}

\subsection{}
Heuristically, arrangements of low degree should be more likely to have effective
CSM class. This is the case for generic arrangements:

\begin{prop}
Let $\cA$ be a {\em generic\/} arrangement of $d\ge 1$ distinct hyperplanes in~$\Pbb^n$. 
Then $\csm(A)$ is effective for $n=1$ and all $d$, and for for $n>1$ and
\begin{itemize}
\item $n$ even, $d\le n+3$,
\item $n$ odd, $d\le n+4$,
\end{itemize}
and it is not effective otherwise.
\end{prop}

\begin{proof}
The arrangement $\cA$ is generic precisely when $A$ is a divisor with simple normal
crossings. As seen in Example~\ref{norcro}, the CSM class of the complement is
$\csm(M(\cA))=\frac 1{(1+h)^d}\cap (c(T\Pbb^n)\cap [\Pbb^n])$, and hence
$\csm(A)$ equals{\small
\[
\left (1-\frac 1{(1+h)^d}\right) (1+h)^{n+1} \cap [\Pbb^n]
=\sum_{k=0}^n \left(\binom{n+1}k - (-1)^k \binom{k+d-n-2}k\right) h^k\cap[\Pbb^n]
\quad.
\]}
The statement is easy to verify from this expression. If $n$ is even
and $d\ge n+4$, the coefficient of $[\Pbb^0]$ (i.e., the Euler characteristic of $A$) is
bound by
\[
(n+1)-\binom{n+2}2 < 0\quad;
\]
if $n>1$ is odd and $d\ge n+5$, the coefficient of $[\Pbb^1]$ is bound by
\[
\binom{n+1}2 - \binom{n+2}3<0\quad,
\]
so the class is not effective outside of the specified range.\end{proof}

\begin{example}
The smallest generic arrangement with non-effective CSM class consists of
six lines in $\Pbb^2$:
\begin{center}
\includegraphics[scale=.5]{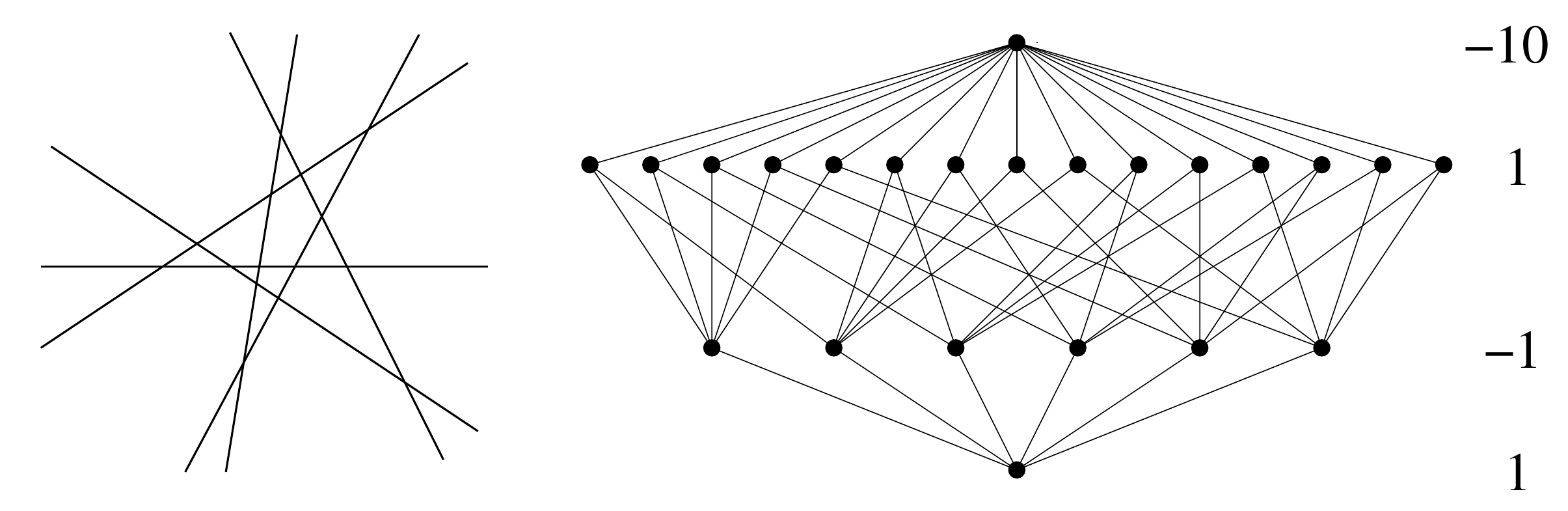}
\end{center}
The polynomial in the statement of Proposition~\ref{effe} is
\[
-(6\cdot (-1)\cdot (t+1)^2 + 15\cdot 1\cdot (t+1) -10)
=6 t^2 -3 t+1\quad,
\]
and not all its coefficients are nonnegative. The CSM class of this arrangement
is $6[\Pbb^1]-3[\Pbb^0]$; its Euler characteristic is $-3$.
\qede
\end{example}

\subsection{}\label{posifree}
Another source of interesting examples comes from free arrangements. 

\begin{prop}\label{freeposi}
For $n\le 8$, every free arrangement $\cA$ of $d\le n$ hyperplanes in $\Pbb^n$ has
effective CSM class.
\end{prop}

\begin{proof}
According to a theorem of Terao (\cite{MR608532}; see also \cite{MR1217488},
Chapter~4), the characteristic polynomial of a free central arrangement $\hcA$
in $k^{n+1}$ factors over $\Zbb$:
\[
\CP_{\hcA}(t)=(t-d_1)\cdots (t-d_n)\cdot (t-d_{n+1})\quad,
\]
where the $d_i$'s are the `exponents' of the arrangement, i.e., the 
degrees of the generators of the (free) module of $\hcA$-derivations
(Definitions~4.5 and~4.25 in~\cite{MR1217488}). One of the exponents
necessarily equals~$1$ (cf.~\S\ref{genera}); the sum of the exponents
equals the number of hyperplanes in the arrangement (\cite{MR1217488},
Proposition~4.26).

Thus, we may assume that the characteristic polynomial of the arrangement is
\[
(t-d_1)\cdots (t-d_n) (t-1)
\]
with $n\le 8$, $d_i\in \Nbb$, $d_1+\cdots + d_8 \le n-1$. With the aid of a 
computer, applying Proposition~\ref{effe} to all these cases is straightforward.
\end{proof}

Based on Proposition~\ref{freeposi} and the case of generic arrangements, one
may be tempted to guess that arrangements of $d\le n$ hyperplanes in $\Pbb^n$
have effective CSM class. The following is the smallest counterexample
to this statement, for free arrangements.

\begin{example}\label{counter}
The polynomial
\[
\left|
\begin{matrix}
x_0 & x_0^3 & x_0^5 \\
x_1 & x_1^3 & x_1^5 \\
x_2 & x_2^3 & x_2^5
\end{matrix}
\right|
= x_0 x_1 x_2 (x_0-x_1)(x_0-x_2)(x_1-x_2)(x_0+x_1)(x_0+x_2)(x_1+x_2)
\]
defines a free arrangement of $9$ lines in $\Pbb^2$, with exponents $1$, $3$, $5$.
The corresponding $9$ planes in $k^3$ meet along $13$ distinct lines; $6$ of these
lines lie on $2$ planes, $4$ on $3$, and $3$ on $4$ planes. It follows that the 
M\"obius function takes values $1$ at $6$ lines, $2$ at $4$, and $3$ at $3$. 
It also follows that the value of the M\"obius function for the affine arrangement
at the origin is $-15$.
\begin{center}
\includegraphics[scale=.5]{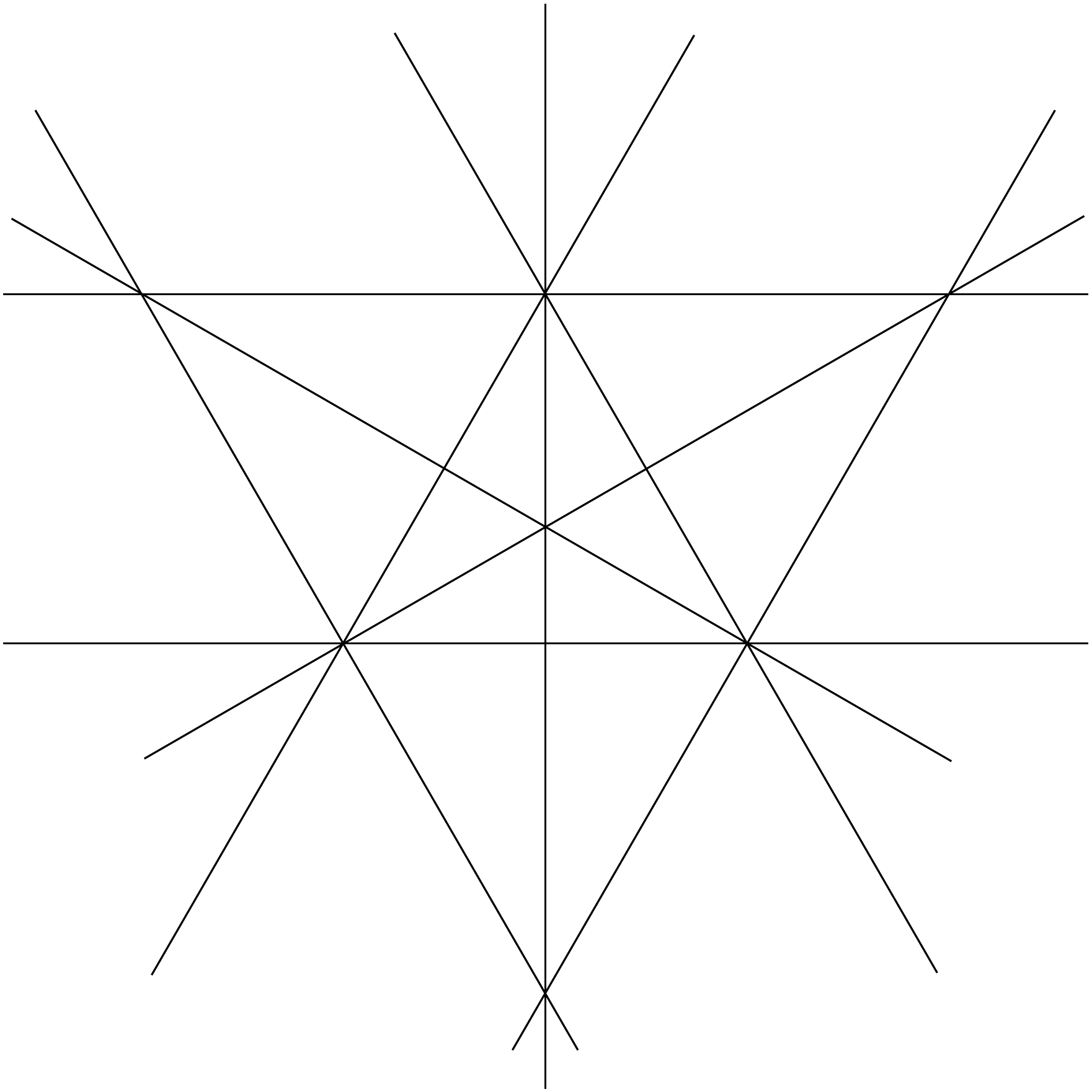}
\end{center}

The cone over this projective arrangement in $\Pbb^9$ is a free arrangement $\cA$
of $9$ hyperplanes, with characteristic polynomial
\[
\CP_{\hcA}(t)=t^{10}-9t^9+(6\cdot 1+4\cdot 2+3\cdot 3) t^8-15t^7= (t-5)(t-3)(t-1)t^7\quad.
\]
Using the criterion in Proposition~\ref{effe}, we compute
\begin{multline*}
-(-9 (t+1)^9 +  23(t+1)^8 -15 (t+1)^7) \\
= 9 t^9+58 t^8 + 155 t^7+217 t^6+161 t^5 +49 t^4 \underline{-7 t^3-5 t^2}+2 t+1
\quad,
\end{multline*}
verifying that the Chern--Schwartz--MacPherson class of this arrangement is not effective.

We end by remarking that the CSM routine described in \cite{MR1956868} and
implemented in Macaulay2 offers a quick verification of this computation: the 
calculation of this CSM class from the equation of the arrangement takes
about $.1$ seconds on a laptop computer. The same routine may be used to 
compute the Segre class of the singularity subscheme, giving with notation as in 
Theorem~\ref{segrethm}
\[
(\sigma_0,\dots,\sigma_9)=(1,0,-49,664,-6528,54272,-389120,2260992,-7340032,
-58720256).
\]
Applying Theorem~\ref{segrethm} gives then
\[
\underline{\pi_\hcA}(t)=1+8t+15t^2=(1+3t)(1+5t)\quad,
\]
in agreement with the combinatorial computation shown above.
\qede
\end{example}


\end{document}